\documentclass[10pt]{amsart}

\usepackage{amsmath, amscd, amssymb, euler}
\usepackage[frame,cmtip,arrow,matrix,line,graph,curve]{xy}
\usepackage{graphpap, color,bbm, verbatim}
\usepackage[mathscr]{eucal}
\usepackage[normalem]{ulem}
\usepackage{pgf,tikz}
\usepackage{multirow,array}
\usetikzlibrary{arrows}

\numberwithin{equation}{section}

\newcommand{\CC}{\mathbb{C}}

\newcommand{\PP}{\mathbb{P}}

\newcommand{\bM}{\mathbf{M}}

\newcommand{\bQ}{\mathbf{Q}}

\newcommand{\cal}{\mathcal}

\def\cC{{\cal C}}

\def\cE{{\cal E}}
\def\cF{{\cal F}}

\def\cK{{\cal K}}
\def\cL{{\cal L}}

\def\cO{{\cal O}}

\def\cU{{\cal U}}

\def\cI{{\cal I}}

\newcommand{\ses}[3]{0\lr{#1}\lr{#2}\lr{#3}\lr 0}

\def\lr{\rightarrow}

\input xy
\xyoption{all}

\DeclareMathOperator{\Hom}{Hom} 

\newtheorem{prop}{Proposition}[section]
\newtheorem{theo}[prop]{Theorem}
\newtheorem{lemm}[prop]{Lemma}
\newtheorem{coro}[prop]{Corollary}

\theoremstyle{definition}
\newtheorem{exam}[prop]{Example}
\newtheorem{defi}[prop]{Definition}
\newtheorem{rema}[prop]{Remark}

\newcommand{\opt}{\mathcal{O}_{\mathbb{P}^2}}
\newcommand{\oo}{\mathcal{O}}
\newcommand{\pt}{{\mathbb{P}^2}}
\newcommand{\pp}{\mathbb{P}}

\newcommand{\bN}{\mathbf{N}_d}

\title{The geometry of the moduli space of one-dimensional sheaves}

\author{Jinwon Choi}
\address{School of Mathematics, Korea Institute for Advanced Study, Seoul 130-722, Korea}
\email{jwchoi@kias.re.kr}

\author{Kiryong Chung}
\address{School of Mathematics, Korea Institute for Advanced Study, Seoul 130-722, Korea}
\email{krjung@kias.re.kr}


\begin{document}
\begin{abstract}
 Let $\mathbf{M}_d$ be the moduli space of stable sheaves on $\mathbb{P}^2$ with Hilbert polynomial $dm+1$. In this paper, we determine the effective and the nef cone of the space $\mathbf{M}_d$ by natural geometric divisors. Main idea is to use the wall-crossing on the space of Bridgeland stability conditions and to compute the intersection numbers of divisors with curves by using the Grothendieck-Riemann-Roch theorem. We also present the stable base locus decomposition of the space $\mathbf{M}_6$. As a byproduct, we obtain the Betti numbers of the moduli spaces, which confirm the prediction in physics.
\end{abstract}

\maketitle

\section{Introduction}
In the middle of 1990's, many algebraic geometers have studied the ample cone of moduli space of vector bundles, or more generally, torsion free sheaves on a projective variety. 
 Recently in \cite{bayer}, the authors provided a canonical way to determine effective and nef divisors of moduli space of semistable objects in a derived category. In this paper, we study the nef cone of the moduli space of one-dimensional sheaves on the projective plane.

Let $\bM_d:=\bM(d,1)$ be the moduli space of stable sheaves with Hilbert polynomial $dm+1$ on $\PP^2$. We find the extremal rays of the nef cone of $\bM_d$ in term of geometric divisors arising from the Fitting map (\cite{lepot1}) and the wall crossing studied in \cite{cc1}. In \cite{lepot1}, Le Potier described the generators of the Picard group of $\bM_d$ by the group homomorphism
$$
\lambda:\mbox{K}(\PP^2)\lr \mbox{K}(\PP^2\times \bM_d)\lr \mbox{K}(\bM_d)\lr \mbox{Pic}(\bM_d)
$$
defined by $\lambda(k)=\mbox{det}(p_{!}(\cF\otimes q^*k))$ for $k\in \mbox{K}(\PP^2)$, $\cF$ is a universal sheaf on $\PP^2\times \bM_d$, and $p$ (resp. $q$) is the projection map $\PP^2\times \bM_d$ onto the first (resp. second) factor. In general, the homomorphism $\lambda$ is restricted to the orthogoal in $K(\PP^2)$ of the class of the sheaves in $\bM_d$, so that $\lambda$ is independent of choice of the universal sheaf (\cite[Lemma 8.1.2]{HL}).
For the basic properties of this homomorphism, see \cite[Chapter 8]{HL}.   We have two divisors
$$
A:=\lambda(\cO_{p}), D:=\lambda(-d\cO_{\PP^2} +\cO_{l})
$$
which generate the Picard group of the space $\bM_d$, where $p$ is a point and $l$ is a line. Geometrically, the divisor $A$ is supported on the locus of stable sheaves whose support contains a fixed point $p$ in $\PP^2$.
The divisor $D$ is a generalization of the \emph{theta} divisor on the Jacobian variety of a smooth projective curve (Lemma \ref{jac}).
Unlike the moduli space of vector bundle on a smooth curve, the Brill-Noether locus, i.e. the locus of stable sheaves $F$ with $h^0(F)\geq2$ in our moduli space is no more a divisor (c.f \cite[Proposition 4.2]{cc1.5}). One can instead consider another geometric divisor arising from the relative Hilbert scheme of points. From \cite{cc1}, we know that the moduli space $\bM_d$ is birational with the relative Hilbert scheme of $g=\frac{(d-1)(d-2)}{2}$ points over the universal family $\cC\longrightarrow |\cO_{\PP^2}(d)|$ of degree $d$ curves.
\begin{defi}
Let $L$ be the divisor in $\bM_d$ corresponding to the divisor of the relative Hilbert scheme consisting of pairs $(C,Z)$ such that $Z$ passes through a fixed line in $\PP^2$.
\end{defi}
One can give a scheme theoretic definition of the divisor $L$ by using the proof of \cite[Proposition 1.1]{yuan2}. It turns out that the divisor $L$ is an extremal ray of the effective cone of the moduli space $\bM_d$ (Lemma \ref{eff}).
The main theorem of this paper is the following.
\begin{theo}\label{mainthm}
The two extremal rays of the nef cone of the moduli space $\bM_d$ are generated by the divisors $A$ and
$$
B=\left\{
    \begin{array}{ll}
      \frac{(d-2)^2(d+2)}{8}A+L, & \text{if } d \text{ is even}; \\
      \frac{(d-1)(d+4)(d-3)}{8}A+L, & \text{if } d \text{ is odd}
    \end{array}
  \right.
$$
\end{theo}
The proof of Theorem \ref{mainthm} is as follows. The divisor $A$ is the pull-back of the hyperplane divisor over the complete linear system $|\cO_{\PP^2}(d)|$ along the Fitting map $\bM_d\lr |\cO(d)|$. Thus the divisor $A$ is nef. To find the other ray, firstly, we express the determinant line bundle $D$ as a linear combination of the geometric divisors $A$ and $L$. Secondly, by using the result in \cite{woolf}, we describe the first wall-crossing center in the sense of Bridgeland in terms of the relative Hilbert scheme of points. Then by the result of Bayer and Macr\`i (\cite{bayer}), we find a numerical divisor $B$ as the other ray of the nef cone.
\begin{rema}
Like the case of the moduli space of vector bundles, it is meaningful to find the geometric meaning of the \emph{numerical} divisor $B$.
\end{rema}
Our computation of the nef cone of the moduli space $\bM_d$ is essentially the same as that of Woolf (\cite{woolf}). We apologize for any overlap with \cite{woolf} although our motivation is totally independent.

In \S 3, by using the Bridgeland wall-crossing, we find the stable base locus decomposition and compute the Poincar\'e polynomial of the moduli space $\bM_6$ (Table \ref{tab:M6} and Theorem \ref{thm:m6}). The Poincar\'e polynomials of $\bM_d$ have been predicted in physics by means of PT-BPS correspondence and the B-model computation of refined BPS index \cite{ckk,hkk}. When $d\le 5$, it has been rigorously checked by several methods \cite{cc1,yuan2,cm,maicantorus}. In the authors' knowledge, our result for $d=6$ is new.

\medskip
\textbf{Acknowledgement.}
We would like to thank D. Chen, S. Katz, H.-B. Moon, K. Yoshioka and M. Woolf for helpful discussions and comments. The first author was partially supported by TJ Park Science Fellowship of POSCO TJ Park Foundation and the second author was partially supported by National Research Foundation of Korea (Grant No. 2013R1A1A2006037).

\section{Proof of Theorem \ref{mainthm}}
\subsection{Effective cone of the moduli space $\bM_d$}
\label{sec:wd}
As a birational model of the moduli space $\bM_d$, one can consider a projective bundle over a Kronecker modules space. Let $\bN:=N(3;d-1,d-2)$ be the parameter space of sheaf homomorphisms
$$
\cO_{\PP^2}(-2)^{\oplus d-1}\longrightarrow \cO_{\PP^2}(-1)^{\oplus d-2}
$$
acted on by the automorphism group $\text{GL}(d-1)\times \text{GL}(d-2)/\CC^*$.
Geometrically some open subset of the quotient space $\bN$ parameterizes $g=\frac{(d-1)(d-2)}{2}$ points of general position on $\PP^2$ and thus $\bN$ is birational with Hilbert scheme of $g$ points.
In terms of quiver, this is the moduli space of quiver representations of 3-Kronecker quiver
$$
\xymatrix{\bullet\ar@/^/[r] \ar@/_/[r]\ar[r]&\bullet\\}
$$
with dimension vector $(d-1, d-2)$.
\begin{defi}
 A representation of 3-Kronecker quiver with dimension vector $(e,f)$ is a pair of vector spaces $E$ and $F$ of dimension $e$ and $f$ respectively and three maps $\phi_1$, $\phi_2$, and $\phi_3$ from $E$ to $F$. A representation $(E,F,\phi_i)$ is called \emph{semistable} if there is no pair of subvector spaces $E'\subset E$ and $F'\subset F$ with dimensions $e'$ and $f'$ such that $\phi_i(E') \subset F'$ for all $i$ and
 $$ e'f-ef' >0 \text{  (or }\frac{e'}{f'}>\frac{e}{f}).  $$
\end{defi}
It is well known that the GIT-quotient $\bN$ is smooth and its Picard group is free of rank one.
Also $\bN$ carries a universal sheaf homomorphism
$$
\tau:p^*\cE \otimes q^*(\cO_{\PP^2}(-2))\lr p^*\cF \otimes q^*(\cO_{\PP^2}(-1))
$$
on $\bN\times \PP^2$ such that $p:\bN\times \PP^2\lr \bN$ and $q:\bN\times \PP^2\lr \PP^2$ are the projection onto its factors.
Let
$$
\cU:=p_*(\mbox{Coker}(\tau^*))
$$
be the push-forward sheaf on $\bN$. Since $h^0(\cU|_{\{n\}\times \PP^2})$ does not depend on $n \in \bN$, the sheaf $\cU$ is locally free of rank $3d$ over $\bN$ (cf. \cite[Proposition 3.2.1]{maican5}). Let $\bQ_d=\PP\cU$ be the projectivization of $\cU$. Then $\bQ_d$ is a birational model of the moduli space $\bM_d$. Following the idea of the proof of \cite[Proposition 1.1]{yuan2}, let
$$
\tau^*|_{n}:\cO_{\PP^2}(1)^{\oplus d-2}\lr \cO_{\PP^2}(2)^{\oplus d-1}
$$
be the dual homomorphism representing a point $n\in \bN$. For the general $n$, the cokernel $Q_n$ of $\tau^*|_{n}$ is torsion free and thus the cokernels of the homomorphism $ \cO_{\PP^2}\lr Q_n$ are stable sheaves. By taking the dual $\mathcal{E}xt^1(-,\cO_{\PP^2})$ for each sheaf, we obtain a stable sheaf in $\bM_d$. More strongly, we have
\begin{prop}\label{codim1}
The spaces $\bM_d$ and $\bQ_d$ are isomorphic with each other up to codimension two and thus $\mbox{Eff}(\bM_d)=\mbox{Eff}(\bQ_d)$.
\end{prop}
\begin{proof}
As we mentioned in the introduction, the space $\bM_d$ is isomorphic to a projective bundle over an open subset of the Hilbert scheme of points up to codimension one. The complement of this open subset in the Hilbert scheme is explicitly described in \cite[Proposition 1.1]{yuan2}. It is easy to check that this open set is isomorphic to an open set in $\bN$ whose complement has codimension at least two. Hence we obtain the result.
\end{proof}

The Picard group of $\bN$ is generated by the divisor class $L_0$ consisting of $g$ points meeting a fixed line in $\PP^2$. Let us denote
$L'$ by the pull-back of $L_0$ along the projective bundle map $\bQ_d\lr \bN$. By Proposition \ref{codim1}, the divisor $L'$ can be naturally identified with the divisor $L$ on $\bM_d$. From now on, we will write $L'$ as $L$.
\begin{coro}\label{eff}
The divisors $A$ and $L$ are generators of the extremal rays of the effective cone of the moduli space $\bM_d$.
\end{coro}
\begin{proof}
By the construction, it is clear that the divisor $A$ is an extremal ray of the effective cone of the space $\bM_d$ since the Fitting map $\bM_d\lr |\cO(d)|$ is well-defined.

On the other hand, $\mbox{Eff}(\bM_d)=\mbox{Eff}(\bQ_d)$ by Proposition \ref{codim1}. But $L$ is one of extremal rays of the effective cone of $\bQ_d$ since $\bQ_d$ is a projective bundle over $\bN$. Thus we have the result.
\end{proof}
\begin{rema}
In the Bridgeland wall-crossing, an extremal ray of effective cone of $\bM_d$ corresponds to the collapsing wall at which destabilizing object is $\opt$.
Let
$$
\chi (E,F)=\int_{[\PP^2]} \mbox{ch}(E)\mbox{ch}(F) \mbox{Td}(\PP^2)
$$
be the Euler form for $E, F \in \mbox{Coh}(\PP^2)$.
We let $v$ be the class in $K(\PP^2)$ of sheaves in $\bM_d$ and $v'=\mbox{ch}(\opt)$. Then by the general theory of Bridgeland wall-crossing \cite{abch, bertram}, the divisor corresponding the wall is given by $\lambda(w)$ for  $w\in K(\PP^2)$ such that $\chi(w,v)=\chi(w,v')=0$.  By direct calculation, one can see $\lambda(w)=\lambda(-d\cO+\cO_l)+(d-1)A$, which is numerically equivalent to $L$ by Proposition \ref{dinterms}.
\end{rema}

\subsection{Determinant line bundles in terms of geometric divisors}
Let us recall the two divisors $A:=\lambda(\cO_{p})$, $D:=\lambda(-d\cO + \cO_l)$ for a line $l$ and a point $p$. These two divisors freely generate the Picard group of the moduli space $\bM_d$ (\cite{lepot1}).
Note that the divisor $A$ is the pulling-back divisor of the divisor $\cO(1)$ in the complete linear system $|\cO_{\PP^2}(d)|$ along the Fitting map $\bM_d\lr |\cO(d)|$. Recall that the divisor $L$ is the divisor class consisting of pairs of a degree $d$ curve $C$ and a subscheme $Z$ of length $g:=\frac{(d-1)(d-2)}{2}$ such that $Z$ passes through a fixed line. For later use,
we remark that a pair $(C,Z)$ corresponds to the sheaf $\cE xt^1(I_{Z,C}(d-3),\cO(-3))$. In this subsection, we express $D$ as a linear combination of $A$ and $L$. This is obtained by intersecting with test curves.
\begin{prop}\label{dinterms}
Under above definition and notations, there exists a numerical equivalence
$$D=(1-d)A+L.$$
\end{prop}
\begin{proof}
Choose very general $g$ points on $\PP^2$. Let $P\simeq \PP^1$ be a pencil of degree $d$ curves with fixed $g$ base points. By identifying an element in $P$ with a stable sheaf, we may consider $P$ as a curve in $\bM_d$. Then, clearly we have
$$
P\cdot A=1,\hspace{1em} P\cdot L=0.
$$
On the other hand, let $S$ be the blow-up $\PP^2$ along $d^2$ base points where the $d^2$ points are the intersection points of two curves of degree $d$.
Then $S$ is given by a hypersurface of bi-degree $(1,d)$ in $P\times \PP^2\simeq \PP^1\times \PP^2$. Let $E_i$ be the exceptional divisors on $S$ for $1\le i\le d^2$.
Hence
$$
S\hookrightarrow P\times \PP^2.
$$
Let $\bigcup_{i=1}^{g} E_i=Z$.
From the structure sequence,
$$
\ses{I_{Z,S}}{\cO_S}{\cO_Z}
$$
By taking $\mathcal{E}xt^1(-, \cO(-3))$ followed by twisting $\opt(d-3)$, we obtain
$$
\ses{\mathcal{E}xt^1(\cO_S(d-3), \cO(-3))}{\mathcal{E}xt^1(I_{Z,S}(d-3), \cO(-3))}{\mathcal{E}xt^2(\cO_Z(d-3), \cO(-3))}.
$$
Let $\cF:=\mathcal{E}xt^1(I_{Z,S}(d-3), \cO(-3)))$ be the extension class. Since $E_i$ is of general position by construction, the sheaf $\cF$ is a flat family of stable sheaves with Hilbert polynomial $dm+1$ parameterized by $\PP^1$ (\cite[Lemma 12]{maicandual}).
To compute the Chern character $\mbox{ch}(\cF)$, let us find the free resolution of $\mathcal{E}xt^1(\cO_S(d-3), \cO(-3)))$. From the resolution of $\cO_S$ we obtain
$$
\ses{\cO(0,-d)}{\cO(1,0)}{\mathcal{E}xt^1(\cO_S(d-3), \cO(-3)))}.
$$
Thus,
$$
\mbox{ch}(\cF)=e^{p}-e^{-dh}+gh^2.
$$
Here, $p$ is the point class in the Chow group $A^*(P)$ and $h$ is the line class in the Chow group $A^*(\PP^2)$.
Also,
$$
\mbox{Td}(P\times \PP^2/P)\cdot \mbox{ch}(-d+\cO_l)=(1+\frac{3}{2}h+h^2)\cdot (-d+1-e^{-h})=-d+(1-\frac{3}{2}d)h+(1-d)h^2.
$$
Hence, by the Riemann-Roch Theorem and the base change property of the determinant line bundles, we obtain
\begin{align*}
D\cdot P
&= \mbox{Coeffi}_{ph^2}[\mbox{ch}(\cF)\mbox{Td}(P\times \PP^2/P)\mbox{ch}(-d+\cO_L)])\\
&= \mbox{Coeffi}_{ph^2}[(e^p)(-d+(1-\frac{3}{2}d)h+(1-d)h^2)])\\
&= 1-d.
\end{align*}

Now let us compute the intersection numbers with another test curve. Fix a general smooth degree $d$ curve $C$ in $\PP^2$. Let $T\cong C$ be the curve in $\bM_d$ by varying one point on $C$ while fixing general $g-1$ points. Clearly we have
$$
A\cdot T=0.
$$
To compute $D\cdot T$, we construct a family of stable sheaves as follow. Let
$$
j: C\times C \lr C \times \PP^2
$$
be the natural inclusion. Let $p_1=\{pt\}\times C$, $p_1= C\times \{pt\}$ and $p=\{pt\}\times \PP^2$. Let $\Delta \subset C \times C$ be the diagonal. Then,
$$
j_*1=dh, \hspace{1em} j_*p_1=dph,\hspace{1em}  j_*p_2 =h^2,\hspace{1em}  j_*\Delta=dph+h^2.
$$
Let $F=j_*\cO_{C\times C}((g-1)p_2+\Delta)$. Then $F$ fits into the short exact sequence $$\ses{j_*\cO_{C\times C}}{F}{j_*\cO_Z}$$ where $Z= \Delta \cup C\times \{(g-1) \mbox{pts}.\} \subset C\times C$.
Now, $$j^*\frac{\mbox{Td} C}{\mbox{Td} \PP^2}=j^*(\frac{1+(1-g)p_2}{1+\frac{3}{2}h+h^2})=\frac{1+(1-g)p_2}{1+\frac{3}{2}dp_2}=1+(1-g-\frac{3}{2}d)p_2.$$
Hence,
$$
\mbox{ch} j_*\cO_{C\times C}=j_*(\mbox{ch}(1\cdot j^*\frac{\mbox{Td} C}{\mbox{Td} \PP^2}))= dh-\frac{1}{2}d^2h^2,
$$
$$
\mbox{ch} j_*\cO_Z = j_*((\Delta +(g-1)p_2)(1+(1-g-\frac{3}{2}d)p_2))=dph+gh^2 +(1-g-\frac{3}{2}d)ph^2.
$$
That is,
$$
\mbox{ch}( F)= dh+(dph +(g-\frac{1}{2}d^2)h^2)+(1-g-\frac{3}{2}d)ph^2.
$$
Hence,
$$
D\cdot T= \mbox{Coeffi}_{ph^2}[\mbox{ch}(F)\mbox{Td}(T\times \PP^2/T)\mbox{ch}(-d+\cO_l)])=dg.
$$
Finally, we will see $L\cdot T=dg$ by describing the sheaves in $T$ geometrically as follows. We start with an example where $d=4$.

\begin{exam} \label{ex:deg4}
Let $C$ be a smooth quartic curve. A sheaf $F$ in $\bM_4$ whose support is $C$ is $\cO_C(p_1+p_2+p_3)$ where $p_1,p_2$, and $p_3$ are points in $C$. By Serre duality, we have $H^1(F)\simeq Hom(F, \cO_C(1))$. Hence $h^1(F)>0$ if and only if $p_1,p_2$, and $p_3$ are collinear.  If $p_1,p_2$, and $p_3$ are not collinear, we have a unique global section of $F$ whose zeros are these three points. Hence the sheaf $F$ is completely determined by $p_1,p_2$, and $p_3$. Suppose now that $p_1,p_2$, and $p_3$ are collinear and let $\ell$ be the line containing $p_1,p_2$, and $p_3$. If we let $p$ be the fourth point in $C\cap \ell$, it is easy to see that $F=I_{p,C}(1)$. Therefore, $F$ is determined by $p$.

With this description, it is easy to see that $L\cdot T=12$ when $d=4$.
\end{exam}

\begin{lemm}\label{lem:sheafequi}
  Let $C$ be a smooth plane curve of degree $d$ and $g$ be its genus. Fix general $g-1$ points $Z'$ on $C$ and let $p$ be a varying point on $C$. Denote $Z=Z'\cup \{p\} $. 
  Let $E$ be the unique curve of degree $d-3$ passing through fixed $g-1$ points. Then, we have the following.
\begin{enumerate}
\item If $p$ does not lie on $E$, then the pair $(C,Z)$ uniquely determines the sheaf $F$ in $\bM_d$ by the extension
    \[\ses{\cO_C}{F}{\cO_Z}.\]
\item If $p$ lies on $E$, then the corresponding sheaf $F$ is isomorphic to the ideal sheaf of ${E\cap C\setminus Z}$ in $C$ twisted by $(d-3)$. Moreover, $F$ is uniquely determined by ${E\cap C\setminus Z}$, that is, if we write $F$ as \[\ses{\cO_C}{F}{\cO_{Z'}},\] then there is another degree $d-3$ curve $E'$ such that ${E'\cap C\setminus Z'}={E\cap C\setminus Z}$.
\end{enumerate}

\end{lemm}

\begin{proof}
The proof is similar to Example \ref{ex:deg4}. We first note that $Z$ lies on a degree $d-3$ curve $E$ if and only if there is a nonzero map $F\to \cO_C(d-3)$. Indeed, from the exact sequence
\[\ses{\cO_C}{\cO_C(d-3)}{\cO_{E\cap C}},\]
the inclusion $\cO_Z\to \cO_{E\cap C}$ induces a map $F\to \cO_C(d-3)$ which factors through the inclusion $\cO_C\to \cO_C(d-3)$.

 Since $g-1 = \frac{d(d-3)}{2}$, $Z'$ determines a unique degree $d-3$ curve $E$. If the moving point $p$ does not lie on $E$, then the corresponding sheaf $F$ has $Hom(F, \cO_C(d-3))\simeq H^1(F)=0$ and hence $h^0(F)=1$. So, $Z'$ determines a unique sheaf.

Suppose now that $p$ lies on $E$. Then we have a nonzero map $i\colon F\to \cO_C(d-3)$, which is necessarily injective as $C$ is smooth. So $F$ fits into an exact sequence
\begin{equation}\label{eq:ibcz}\ses{F}{\cO_C(d-3)}{\cO_{E\cap C\setminus Z}}.\end{equation}
Therefore $F$ is the ideal sheaf of ${B\cap C\setminus Z}$ in $C$ twisted by $(d-3)$.

To show the converse, it is enough to show that $F$ in \eqref{eq:ibcz} has $h^1(F)=1$, because then any map $F\to \cO_C(d-3)$ is a constant multiple of the map in \eqref{eq:ibcz} and hence ${E\cap C\setminus Z}$ uniquely determines $F$. From \eqref{eq:ibcz}, we see that $h^0(\cO_C(d-3))=g$ and $\cO_{E\cap C\setminus Z}=g-2$. Since $Z'$ is chosen sufficiently general and $\cO_C(d-3)$ is generated by its global sections, we conclude that $h^0(F)=2$ or equivalently $h^1(F)=1$.
\end{proof}

Now we compute the intersection number $L\cdot T$. Fix a general line $l$. A sheaf $F$ is in $L\cap T$ if $F$ can be written as
\[\ses{\cO_C}{F}{\cO_Z}\]
where $Z$ intersects $l$. By Lemma \ref{lem:sheafequi}, there are two possibilities: the varying points $p$ can either be away from $E$ or on $E$. For $p$ away from $E$, $F$ is in $L\cap T$ if and only if $p$ is one of $d$ points in $C\cap l$. Consider the case where $p$ is on $E$. Note that $E\cap C$ consists of $d(d-3)=2g-2$ points. Since $g-1$ of them are fixed, the remaining $g-1$ points are what $p$ can be. For each of these possibilities, the remaining $(2g-2)-g=g-2$ points uniquely determine a sheaf $F$. But these sheaves must be counted with multiplicity. If we take any point on $C\cap l$, this point together with fixed $g-1$ points give a degree $d-3$ curve $E'$, which by Lemma \ref{lem:sheafequi} gives the same sheaf $F$. Hence the multiplicity is $d$ as there are $d$ possible choices in $C\cap l$. Thus we conclude that
\[L\cdot T = d + d(g-1) = dg.\]

This proves Proposition \ref{dinterms}.
\end{proof}
\medskip

Finally, we give an alternative geometric proof of the intersection number $D\cdot T$. Let $\mbox{Jac}^g(C)$ be the moduli space of line bundles of degree $g$ on a smooth projective curve $C$ with $g(C)=g$. For a fixed $\mathcal{L}\in \mbox{Jac}^{-1}(C) $, let
$$
\Theta:=\{M\in \mbox{Jac}^g(C)| h^0(M\otimes \mathcal{L})\neq 0 \}
$$
be the \emph{theta} divisor on $\mbox{Jac}^g(C)$ with respect to $\mathcal{L}$.
Now we fix a smooth plane curve $C$. Then clearly we have $\mbox{Jac}^g(C)\subset \bM_d$. Let $\mathcal{L}=\cO_C(-p)$ for a $p\in C$. The divisor $\Theta$ on $\mbox{Jac}^g(C)$ parameterizes the divisors on $C$ consisting of $g-1$ freely varying points and the fixed point $p$.
Furthermore,
\begin{lemm}\label{jac}
The restriction of the determinant bundle $D$ to $\mbox{Jac}^g(C)$ is $$ D|_{\mbox{Jac}^g(C)}=d\Theta.$$
\end{lemm}
\begin{proof}
The restriction of the class $-d\cO +\cO_{l}$ to $C$ is $-d\cO_C +d \cO_{p} = -d \cO_C(-p)$. By the same argument as in the example after \cite[Proposition 2.10]{lepot1}, the zero scheme of a section of $\cL(-\cO_C(-p))$ is exactly $\Theta$ defined above. Hence we have $D|_{\mbox{Jac}^g(C)}=d\Theta.$
\end{proof}

Also by definition, the curve $T$ is numerically equivalent to $\frac{1}{(g-1)!}\Theta^{g-1}$ (\cite[Lemma 1]{kou}). Hence we have
$$
D\cdot T=d\Theta \cdot\frac{1}{(g-1)!}\Theta^{g-1}=dg
$$
because $\Theta^{g}=g!$.
\begin{rema}
By using Lemma \ref{lem:sheafequi}, one can compute the intersection number $\Theta\cdot T$ as follows.
  First of all, there is one straightforward element in the intersection, namely, the one where $g$ points are $g-1$ fixed points defining $T$ together with the point $p$ defining $\Theta$. Consider now the elements in $T$ where $g$ points lie on some degree $d-3$ curve $E$. By Lemma \ref{lem:sheafequi}, corresponding sheaf is completely determined by the remaining $g-2$ intersection points of $C$ and $E$. If we take another degree $d-3$ curve $E'$ passing through these $g-2$ points and $p$, since we can take $g-1$ free points defining an element in $\Theta$ to be the remaining $g-1$ intersection points of $C$ and $E'$, the corresponding sheaf is also in $\Theta$. There are $g-1$ possibilities of this case. Hence we see that $\Theta\cdot T=g$.
\end{rema}

\subsection{Nef cone of the space $\bM_d$}
In this section, we determine the nef cone of $\bM_d$ by using the Bridgeland wall-crossing.
We prove that the first wall-crossing locus contains some relative Hilbert scheme of points (\cite{cc1}). We then find a divisor class which is zero on this wall-crossing locus, which complete the proof of Theorem \ref{mainthm}. At the end of this section, we recheck our computation by the technique of Bridgeland wall-crossing.

To determine the net divisor, let us describe the stable sheaves contained in the first wall-crossing center of the moduli space $\bM_d$ (\cite{woolf}). By \cite[Proposition 7.5]{woolf}, we know that the destabilizing objects corresponding to the first wall are
$$
\left\{
  \begin{array}{ll}
I_{Z,\PP^2}(\frac{d-2}{2}) \text{ with }|Z|=\frac{d-2}{2}, & \text{if } d \text{ is even;} \\
\cO(\frac{d-3}{2}), & \text{if } d \text{ is odd}.
  \end{array}
\right.
$$

\begin{lemm}\label{wall}
If $d$ is even, the sheaves of the form $I_{Z,C}(\frac{d-2}{2})$ where $C$ is a degree $d$ curve and $Z$ is a length $\frac{d-2}{2}$ subscheme of $C$ are contained in the first wall-crossing center. If $d$ is odd, the sheaves of the form $\cE xt^1(I_{p,C}, \cO(-3))$ where $C$ is a degree $d$ curve and $p\in C$ is a point are contained in the first wall-crossing center.
\end{lemm}
\begin{proof}

Let $F= I_{Z,C}(\frac{d-2}{2})$. We first show that $F$ is stable. It is enough to show that $I_{Z,C}$ is stable. By taking the dual $\cE xt^1(-,\cO(-3))$ of the short exact sequence $\ses{I_{Z,C}}{\cO_C}{\cO_Z}$ and twisting, we obtain the pair $s: \cO_C \lr \cE xt^1(I_{Z,C},\cO(-3))\otimes \cO(3-d)$ whose cokenel is zero dimensional. Hence this pair can be regarded as an $\infty$-stable pair with Hilbert polynomial $dm+\frac{-d^2+4d-2}{2}$. An $\infty$-stable pair is a pair with zero dimensional cokernel and a $0^+$-stable pair is a stable sheaf with a choice of a nonzero section. Their moduli spaces are related by wall-crossing. For definitions and more details, see \cite{cc1}. By numerically computing the walls of the moduli space of stable pairs with this Hilbert polynomial, one can easily check that there is no wall (cf. \cite[Proposition 4.9]{cc1.5}). Thus the pair $s$ is also a $0^+$-stable pair and hence $\cE xt^1(I_{Z,C},\cO(-3))\otimes \cO(3-d)$ is a stable sheaf. Then by \cite{maicandual}, we see that $F$ is stable.

Twisting by $\cO(\frac{d}{2}-1)$ the short exact sequence $\ses{I_{C,\PP^2}}{I_{Z,\PP^2}}{I_{Z,C}}$, we see that $F$ fits into the exact sequence
$$
\ses{\cO(-\frac{d}{2}-1)}{I_{Z,\PP^2}(\frac{d}{2}-1)}{F}.
$$	
Therefore the destabilizing object of $F$ is $I_{Z,\PP^2}(\frac{d-2}{2})$ where $|Z|=\frac{d-2}{2}$ and thus $F$ lies at the first wall.

When $d$ is odd, $F=\cE xt^1(I_{p,C}, \cO(-3))$ is a stable sheaf by \cite[Lemma 3.1]{iena}. From the structure sequences of $\{p\}\subset C$ and $C \subset \PP^2$, we know that $F$ fits into the exact sequence
$$
0\lr\cO(\frac{-d-3}{2})\lr\cO(\frac{d-3}{2})\lr F\lr \CC_p\lr0.
$$
Hence the destabilizing object of $F$ is $\cO(\frac{d-3}{2})$ and thus $F$ lies at the first wall.
\end{proof}
Let $aD +bA$ be a nef divisor of the moduli space $\bM_d$. To find the ratio of the positive numbers $a$ and $b$, it is suffice to compute the intersection numbers $D\cdot C$ and $A\cdot C$ for the curve $C$ in $\bM_d$ defined by the stable sheaves in Lemma \ref{wall} (\cite[Theorem 1.1]{bayer}).

\begin{prop}
The divisor class which has intersection number zero with the curve class in the wall-crossing locus is given by
$$
\left\{
  \begin{array}{ll}
    D+\frac{d(d^2-2d+4)}{8}A, & \text{if } d \text{ is even} \\
    D+\frac{(d-1)(d^2+d-4)}{8}A, & \text{if } d \text{ is odd}
  \end{array}
\right.
$$
\end{prop}
\begin{proof}
Suppose $d$ is even. By Lemma \ref{wall}, the stable sheaves $I_{Z,C}(\frac{d-2}{2})$ are contained in the first wall-crossing center where $C$ is a degree $d$ curve and $Z$ is a subscheme of $C$ of length $\frac{d-2}{2}$.
Consider any pencil $P$ of degree $d$ curves defined as in Proposition \ref{dinterms}. We use the same notation as before.
From the structure sequence, we have
$$
\ses{I_{Z,S}(\frac{d-2}{2})}{\cO_S(\frac{d-2}{2})}{\cO_{Z}(\frac{d-2}{2})}.
$$
Then $\cF:=I_{Z,S}(\frac{d-2}{2})$ is a universal family over $P$. Similarly as before, we have
$$
\ses{\cO(-1,-d+\frac{d-2}{2})}{\cO(0,\frac{d-2}{2})}{\cO_S(\frac{d-2}{2})}.
$$
Hence
\begin{align*}
D\cdot P&=\mbox{Coeffi}_{ph^2}[\mbox{ch}(\cF)\mbox{Td}(P\times \PP^2/P)\mbox{ch}(-d+\cO_l)])\\
&
=\mbox{Coeffi}_{ph^2}[(-e^{-p+(\frac{-d-2}{2})h})(-d+(1-\frac{3}{2}d)h+(1-d)h^2)])\\
&=-\frac{d(d^2-2d+4)}{8}.
\end{align*}
Since $A\cdot P=1$, we get the result.
\bigskip

For odd $d$, By Lemma \ref{wall} again, the stable sheaves $\cE xt^1(I_{p,C}, \cO(-3))$ are contained in the first wall-crossing center where $C$ is a degree $d$ curve and $p\in C$. Let
$$
\ses{I_{p,S}}{\cO_S}{\CC_p}
$$
be the structure sequence. By taking dual $\mathcal{E}xt^1(-, \cO(-3))$ followed by twisting $\cO(\frac{d-3}{2})$, we get
$$
\ses{\mathcal{E}xt^1(\cO_S(\frac{d-3}{2}), \cO(-3)))}{\mathcal{E}xt^1(I_{p,S}(\frac{d-3}{2}), \cO(-3)))}{\mathcal{E}xt^2(\CC_p(\frac{d-3}{2}), \cO(-3)))}.
$$
Then since
$$
\ses{\cO(0,\frac{-3-d}{2})}{\cO(1,\frac{d-3}{2})}{\mathcal{E}xt^1(\cO_S(\frac{d-3}{2}), \cO(-3)))},
$$
we have
\begin{align*}
  D\cdot P&=\mbox{Coeffi}_{ph^2}[e^{(p+\frac{d-3}{2}h)}(-d+(1-\frac{3}{2}d)h+(1-d)h^2)])
  \\&=-\frac{(d-1)(d^2+d-4)}{8}.
\end{align*}
This completes the proof.
\end{proof}
Combining with $D=(1-d)A+L$ in Proposition \ref{dinterms}, we obtain Theorem \ref{mainthm}. The following remark provides another method to 
prove Theorem \ref{mainthm}.  (cf. \cite[\S 4]{bertram})
\begin{rema}
Since we know the destabilizing objects, one can compute the nef divisor of the moduli space $\bM_d$.

{\bf Case 1}. If $d$ is odd,  from the proof of Lemma \ref{wall},  the destabilizing object in the first wall-crossing is given by
$$
\cO(\frac{d-3}{2}).
$$
Hence the nef divisor on the wall is given by
$$
\lambda(w)=aA+bL
$$
for $w\in K(\PP^2)$ such that $\chi(w,v)=\chi(w,v')=0$ for $v=ch(F)=dh+\frac{-3d+2}{2}h^2, v'=ch(\cO(\frac{d-3}{2}))$. 
By direct computation, we have
$$
w=-d+h+\frac{d(d^2-5)}{8}h^2.
$$
Hence,
\begin{align*}
\lambda(w)=&\lambda(-d+h-\frac{h^2}{2})+\lambda(\frac{h^2}{2}+\frac{d(d^2-5)}{8}h^2)\\
&=\frac{(d-1)(d+4)(d-3)}{8}A+L.
\end{align*}
\bigskip
{\bf Case 2}. If $d$ is even,  from the proof of Lemma \ref{wall} again, the destabilizing object is $v'=\mbox{ch}(I_Z(\frac{d-2}{2}))$. Hence we see that
$$
w=-d+h+(d(1-3/2(d+2)/2+(d+2)^2/8)-3/2+\frac{d+2}{2})h^2.
$$
By similar computation, we get
$$
\lambda(w)=\frac{(d-2)^2(d+2)}{8}A+L.
$$
\end{rema}


\section{Wall-crossing of $\bM_6$}
In this section, we will compute the Poincar\'e polynomial of the space $\bM_6$ by using the Bridgeland wall-crossing. The key issue is to find the destabilizing objects at each wall. 
We have also applied the same method to calculate the Poincar\'e polynomial of $\bM_d, d\leq5$ and $\bM(5,2)$ and rechecked the results in \cite[Corollary 5.2, Corollary 5.3]{cc1}, \cite[Theorem 1.3, Theorem 1.4]{yuan2}, \cite[Theorem 1.1]{cm}. We omit the details as the computation is similar.

In \cite{abch}, the wall-crossing for the Hilbert scheme of $n$ points on $\pt$ is described. We denote by $(Hilb^n)_k$ the $k$-th wallcrossing model of the Hilbert scheme of $n$ point on $\pp^2$. For the notational convenience, we use $\cO$ instead of $\cO_{\pt}$.

\newcommand{\mr}[2]{$\begin{array}{c} #1 \\[-1.5ex] #2 \end{array} $}
\renewcommand{\arraystretch}{1.5}

\begin{table}
\begin{tabular}{|c|c|c|c|c|m{2cm}|}
  \hline
  Wall & \mr{$Destabilizing subobject$}{${\footnotesize (Chern character)}$} & $R$ & Divisor & Generic sheaf & Description\\
  \hline
  $W_5$ & \mr{[\oo(-1) \to \oo\oplus\oo(1)]}{\scriptstyle (1,2,0)} & {$\frac{8}{3}$} & {$L+16A$} & $\cI_{Z_2,C}(2)$ & 10 points on a conic\\
  \hline
  $W_4$ & \mr{[\oo(-2) \to \oo(-2)\oplus\oo(1)]}{\scriptstyle (1,1,\frac{1}{2})} & {$\frac{7}{3}$} & {$L+11A$} & $\cO_C(1)(Z_4)$ & 6 points on a line\\
  \hline
  $W_3$ & \mr{[2\oo(-1) \to 3\oo]}{\scriptstyle (1,2,-1)} & {$\frac{\sqrt{46}}{3}$} & {$L+10A$} & $\cI_{Z_3,C}(2)(Z_1)$ & 9 points on a conic\\
  \hline
  $W_2$ & \mr{[\oo(-1) \to 2\oo]}{\scriptstyle (1,1,-\frac{1}{2})} & {$\frac{\sqrt{31}}{3}$} & {$L+5A$} & $\cI_{Z_1,C}(1)(Z_5)$ & 5 points on a line\\
  \hline
  $W_1'$ & \mr{[\oo(-2) \to 2\oo]}{\scriptstyle (1,2,-2)} & {$\frac{\sqrt{28}}{3}$} & {$L+4A$} & $\cI_{Z_4,C}(2)(Z_2)$ & 8 points on a conic\\
  \hline
  $W_1$ & \mr{[2\oo(-2) \to \oo(-1)\oplus 2\oo]}{\scriptstyle (1,3,-\frac{7}{2})} & {$\frac{5}{3}$} & {$L+3A$} & $\cI_{Z_8,C}(3)$ & 10 points on a cubic\\
  \hline
  $W_0$ & \mr{[\oo]}{\scriptstyle (1,0,0)} & {$\frac{4}{3}$} & {$L$} &  & Collapsing wall\\
  \hline
  \end{tabular}
  \caption{Wall-crossing for $\bM_6$}\label{tab:M6}
\end{table}

%

The Bridgeland walls are described in Table \ref{tab:M6}. Since the sheaf in $\bM_6$ has Chern character $(0,6,-8)$, the potential walls are semicircles centered at $(-\frac{4}{3},0)$ of radius
\[
R=\sqrt{\left(\frac{4}{3}\right)^2+ \frac{6d'+8c'}{3r'} },
\]
where $(r',c',d')$ is the Chern character of destabilizing subobjects. By the same technique as in \cite[\S6]{bertram}, it is elementary to check that at actual wall $r'$ must be $1$ and $R$ must be one of those listed in Table \ref{tab:M6}. Moreover, from the classification of sheaves in $\bM_6$ in terms of the locally free resolution of its sheaves carried out in \cite{maican6}, it is clear that such potential walls are actual walls having destabilizing subobjects listed on the table. We have matched the numbering of the walls with the numbering of the corresponding stratum in \cite{maican6}, that is the locus flipped at wall $W_5$ is the stratum $X_5$ and so on. Note that the stratum $X_1$ is the union of wall-crossing loci for two walls $W_1$ and $W_1'$.

In the last two columns of Table \ref{tab:M6}, geometric description of a generic sheaf in each locus is given. Here, $C$ is a plane sextic curve and $Z_n$ is length $n$ subscheme of $C$. A generic sheaf $F$ in $\bM_6$ fits into an exact sequence
\[\ses{\cO_C}{F}{\cO_Z},\]
where $C$ is a sextic curve and $Z$ is a subscheme of $C$ of length $10$ in general position \cite{cc1}. If $Z$ fails to be in general position, the corresponding sheaf belongs to one of the wall-crossing loci. For example, assume $10$ points lie on some conic. Then by Bezout's theorem, the intersection of this conic and $C$ consists of $12$ points. The corresponding sheaf is determined by the remaining two points. This can be seen by a similar argument as in Lemma \ref{lem:sheafequi}.

By Proposition \ref{codim1}, we know the birational model corresponding to the chamber between $W_0$ and $W_1$ is $\bQ_6$. Recall that $\bQ_6$ is a $\PP^{17}$-bundle over $\mathbf{N}_6$ (\S 2). The Poincar\'e polynomial of the moduli space of Kronecker modules can be computed by the techniques of \cite{eschow} or \cite{drezet}. In \cite{reineke} can be found a recursive formula for the Poincar\'e polynomials, which can be implemented on a Mathematica program.
\begin{theo}
  The Poincar\'e polynomial of $\mathbf{N}_6$ is\\
  \medskip
 $ P(\mathbf{N}_6)= 1+q+3q^2+5q^3+10q^4+14q^5+23q^6+30q^7+41q^8+46q^9+51q^{10}+46q^{11}+41q^{12}+ 30q^{13}+ 23q^{14}+ 14q^{15}+ 10q^{16}+ 5q^{17}+ 3q^{18}+q^{19}+q^{20}.$
\end{theo}

We will compute the Poincar\'e polynomial of $\bM_6$ from that of $\bQ_6$ by analyzing wall-crossing contributions. We denote $E_i^\pm$ the exceptional locus of the moduli spaces near each wall. By Table \ref{tab:M6}, the exceptional loci are certain projective bundles over a product of moduli spaces of Bridgeland stable objects. As the generic elements of these moduli spaces are given by a subscheme of certain length, they are birational models of Hilbert schemes. The birational geometry in terms of Bridgeland stability condition is studied in \cite{abch}, with which we compare our stability condition to figure out which birational model must be used. The same technique has been used in \cite[\S6]{bertram}.

We will only consider two walls $W_1$ and $W_4$, as the computation is similar at other walls.
At $W_1$, elements $F$ in $E_1^+$ fits into an exact sequence
\[\ses{[2\oo(-2) \to \oo(-1)\oplus 2\oo]}{F}{\oo(-3)[1]}.\]
The complex $[2\oo(-2) \to \oo(-1)\oplus 2\oo]$ has Chern character $(1,3,-\frac{7}{2})$, and is generically equivalent to $I_{Z_8}(3)$.

By \cite{abch}, the Bridgeland walls for $Hilb^8$ are the following semicircles $W^8_x$ with center at $(x,0)$ and radius $\sqrt{x^2-16}$.
\[ W^8_{-\frac{17}{2}}, W^8_{-\frac{17}{2}},W^8_{-\frac{15}{2}},W^8_{-\frac{13}{2}},W^8_{-\frac{11}{2}},W^8_{-5},W^8_{-\frac{9}{2}},W^8_{-\frac{25}{6}}, \]
where above the wall $W^8_{-\frac{17}{2}}$ we have the Hilbert scheme and $W^8_{-\frac{25}{6}}$ is the collapsing wall. The superscript $8$ indicates that they are walls for $Hilb^8$. By the elementary check following \cite{abch}, the twisting by $\oo(3)$ moves these walls horizontally by 3. By abusing notations, we also denote by $W_x$ the semicircle with center $(x+3,0)$ and radius $\sqrt{x^2-16}$. Since our wall $W_1$ is the semicircle centered at $(-\frac{4}{3},0)$ of radius $\frac{5}{3}$, one can see that $W_1$ is located in between $W^8_{-\frac{9}{2}}$ and $W^8_{-\frac{25}{6}}$. Hence the corresponding moduli space is $(Hilb^8)_6$, which is the sixth and the final birational model of $Hilb^8$. Later we will show this space is isomorphic to a $Gr(2,9)$-bundle over $\pp^2$. By an elementary calculation we check
\begin{align*}
  Ext^1(\oo(-3)[1],[2\oo(-2) \to \oo(-1)\oplus 2\oo]) \simeq \CC^{20},\\
  Ext^1([2\oo(-2) \to \oo(-1)\oplus 2\oo],\oo(-3)[1]) \simeq \CC^{2}.
\end{align*}
Hence $E_1^-$ is $\pp^1$-bundle over $(Hilb^8)_6$, whereas $E_1^+$ is $\pp^{19}$-bundle over $(Hilb^8)_6$.

The wall $W_4$ is more interesting. The element $F$ in $E_4^+$ fits into an exact sequence
\[\ses{[\oo(-2) \to \oo(-2)\oplus\oo(1)]}{F}{[2\oo(-3) \to \oo(-1)]}.\]
The complex ${[2\oo(-3) \to \oo(-1)]}$ is equivalent to $\cI_{Z_4}^\vee(-5)[1]$, where $\vee$ indicates the derived dual $\mathcal{RH}om(-,\cO)$. Indeed, the generic ideal sheaf $\cI_{Z_4}$ is equivalent to ${[\oo(-4) \to 2\oo(-2)]}$, and by taking the derived dual, twisting by $-5$, and shifting by $1$, we obtain ${[2\oo(-3) \to \oo(-1)]}$. All these operations induce isomorphisms on the moduli spaces, but we must find walls for the Hilbert scheme after taking these operations.

By \cite{abch}, the Bridgeland walls for $Hilb^8$ are the following semicircles $W_x$ with center at $(x,0)$ and radius $\sqrt{x^2-8}$.
\[ W^4_{-\frac{9}{2}}, W^4_{-\frac{7}{2}},W^4_{-3}, \]
where above the wall $W^4_{-\frac{9}{2}}$ we have $Hilb^8$ and $W^4_{-3}$ is the collapsing wall. After the above operations, we find that the centers of walls move to $(-5-x,0)$. For example, in \cite{abch} the wall $W^4_{-\frac{9}{2}}$ corresponds to the destabilizing object $\oo(-1)$. Hence, after taking the above operations, we must use the destabilizing object $\oo(-4)[1]$, whose Chern character is $(r',c',d'):=(-1,4,-8)$. Since the Chern character of the complex ${[2\oo(-3) \to \oo(-1)]}$ is $(r,c,d):= (-1,5,-\frac{17}{2})$, from the formula
\[x= \frac{rd'-r'd}{rc'-r'c} \text{ and } R=\sqrt{x^2-2\frac{cd'-c'd}{rc'-r'c}},\]
one can check $x=-\frac{1}{2}$ and the radius does not change.
We remark that twisting by $n$ moves the walls horizontally by $n$, taking the derived dual flips the walls along the $y$-axis, and shifting does not change the wall.

By abusing notation again, we denote by $W^4_x$ the semicircle with center at $(-5-x,0)$ and radius  $\sqrt{x^2-8}$. Then the wall $W_4$ is in between $W^4_{-\frac{9}{2}}$ and $W^4_{-\frac{7}{2}}$. Hence in this case, the corresponding moduli space is $(Hilb^4)_1$.
Since we have
\begin{align*}
  Ext^1([2\oo(-3) \to \oo(-1)], [\oo(-2) \to \oo(-2)\oplus\oo(1)]) \simeq \CC^{24},\\
  Ext^1([\oo(-2) \to \oo(-2)\oplus\oo(1)],[2\oo(-3) \to \oo(-1)]) \simeq \CC^{6},
\end{align*}
$E_4^-$ is $\pp^5$-bundle over $(Hilb^4)_1$, whereas $E_4^+$ is $\pp^{23}$-bundle over $(Hilb^4)_1$.

The analysis of other walls is similar. We conclude
\begin{align}
P(\bM_6)=&P(Q_6)+(P(\PP^{19})-P(\PP^1))P((Hilb^8)_{6}))  \notag\\
& +(P(\PP^{21})-P(\PP^3))P((Hilb^4)_2\times Hilb^2)) \notag\\
& +(P(\PP^{21})-P(\PP^3))P((Hilb^5)_2\times \PP^2) \label{eq:m6wall}\\
& +(P(\PP^{23})-P(\PP^5))P((Hilb^3)_1\times \PP^2)  \notag\\
& +(P(\PP^{23})-P(\PP^5))P((Hilb^4)_1) \notag\\
& +(P(\PP^{25})-P(\PP^7))P(Hilb^2). \notag
\end{align}

The following is straightforward from the wall-crossing in \cite[\S 8]{abch}.
\begin{lemm}\label{lem:hilb}
  \begin{enumerate}
    \item $P((Hilb^3)_1)=P(Hilb^3) + (P(\pp^0) - P(\PP^3)) P(\PP^2)$.
    \item $P((Hilb^4)_1)=P(Hilb^4) + (P(\pp^1) - P(\PP^4)) P(\PP^2)$.
    \item $P((Hilb^4)_1)=P((Hilb^4)_1) + (P(\pp^0) - P(\PP^3)) P(\PP^2\times \pp^2)$.
    \item $P((Hilb^5)_2)=P(Hilb^5) + (P(\pp^2) - P(\PP^5)) P(\PP^2) +(P(\pp^1) - P(\PP^4)) P(\PP^2\times \pp^2)$.
  \end{enumerate}
\end{lemm}
It remains to find $(Hilb^8)_6$. It has the coinciding wall corresponding to two different destabilizing objects (the wall $W_{-\frac{9}{2}}$ in \cite[p. 39]{abch}). So, the corresponding wall-crossing locus is no longer irreducible. Hence, it is hard to compute the wall-crossing contribution to the Poincar\'e polynomial, because the Poincar\'e polynomial is not a motivic invariant. Compare with \cite[\S 6]{cc1}.

Instead of describing the wall-crossings, we will show directly that the final model of $Hilb^8$ is a $Gr(2,9)$-bundle over $\pp^2$.
More precisely, let $G$ denote the relative Grassmann bundle $Gr(2,\cK^*)$, where $\cK$ is the kernel sheaf of the evaluation map $\cO\otimes H^0(\cO(3))\to \cO(3)$. Then an element in $G$ is interpreted as follows. For a fixed point $p$ in $\pt$, an element in the fiber $Gr(2,9)$ represents a pencil of cubic curves passing through $p$. Since a pencil of cubic curves has 9 base points, by taking 8 points other than $p$, we have a rational map to $Hilb^8$. Conversely, given 8 points in general position, we have a unique pencil of cubics passing through them.

%

\begin{prop}\label{prop:hilb8}
The final birational model $(Hilb^8)_6$ is isomorphic to $G$.
\end{prop}
\begin{proof}
By the above argument, we see that $Hilb^8$ is birational to a $G$.
The correspondence fails to be an isomorphism along the locus in $Hilb^8$ of the eight points not in general position and along the locus in $G$ where the corresponding pencil of cubic curves has a common component. Both of them are obviously of codimension at least two.

Since both of $Hilb^8$ and $G$ have the Picard group of rank two, the effective cone of two spaces are naturally identified. One of the ray generators of the effective cone of the Hilbert scheme corresponds to the Hilbert-Chow morphism. By the description in \cite[\S 10.7]{abch}, the other ray generator of the effective cone is the pull-back to $G$ of the ample divisor of $\PP^2$. Since $Hilb^8$ is a Mori dream space, $G$ is the final birational model of $Hilb^8$.
\end{proof}
\begin{rema}
It is well-known that the general point in $Hilb^8$ has the free resolution $2 \cO(-5) \rightarrow \cO(-4)\oplus 2\cO(-3)$. By the result in \cite{dretrau}, the non-reductive quotient of the sheaf homomorphisms in $ \Hom (2 \cO(-5), \cO(-4)\oplus 2\cO(-3))$ is exactly the $Gr(2,9)$-bundle over $\PP^2$.
\end{rema}

Combining \eqref{eq:m6wall}, Lemma \ref{lem:hilb}, and Proposition \ref{prop:hilb8}, we get the Poincar\'e polynomial of the moduli space $\bM_6$.
\begin{theo}\label{thm:m6}

  $P(\bM_6)=\displaystyle (1+q+4q^2+7q^3+16q^4+25q^5+47q^6+68q^7+104q^8+128q^9+146q^{10}+128q^{11}+104q^{12}+ 68q^{13}+ 47q^{14}+ 25q^{15}+ 16q^{16}+ 7q^{17}+ 4q^{18}+q^{19}+q^{20} )\frac{1-q^{18}}{1-q}.$
\end{theo}

\begin{rema}
  The moduli space $\bM_d$ is believed to be closely related to BPS invariant. In \cite{kkv}, this is explained up to degree five by looking at the Euler characteristics of relative Hilbert scheme together with some error terms. In \cite{cc1}, it is shown that the error term can be understood as a wall-crossing contribution among the moduli spaces of stable pairs. (See also \cite[\S9.3]{ckk}.) When the degree exceeds five, due to the fact that the relative Hilbert scheme is no longer smooth and the existence of walls with more than one destabilizing objects, it is difficult to apply their method.

  In \cite{hkk}, by using holomorphic anomaly equation, the refined BPS indices up to degree 9 have been computed, which gives a prediction for the Poincar\'e polynomial of $M_d$ by a conjectural formula of \cite{ckk}. We remark here that Theorem \ref{thm:m6} coincides with the prediction.

  Mathematically, the Poincar\'e polynomials of $\bM_d$ up to degree 5 have been computed in \cite{cc1,yuan2,cm,maicantorus} by several methods.
\end{rema}

\bibliographystyle{amsplain}

\end{document}